\documentclass[11pt]{article}

\usepackage{amssymb,amsfonts,amsmath}
\usepackage{enumerate}

\usepackage[english]{babel}
\usepackage[cp1250]{inputenc}

\usepackage{algorithm2e}

\usepackage{longtable}
\usepackage{graphicx}
\usepackage{index}
\usepackage{fancyhdr}
\usepackage{lhelp}
\usepackage{optparams}
\usepackage{psfrag}
\usepackage{forloop}
\usepackage{setspace}
\usepackage{tikz}
\usetikzlibrary{calc}
\usepackage{subcaption}
\usepackage{pgffor}
\usepackage{xifthen}

\definecolor{lgray}{gray}{0.75}

\usetikzlibrary{decorations.pathreplacing,automata,calc,positioning}

\newcommand{\diam}{{\rm diam}}

\newcommand{\mptt}[1]{}

\newtheorem{theorem}{Theorem} 

\newtheorem{proposition}[theorem]{Proposition}

\newtheorem{observation}[theorem]{Observation}

\newcommand{\ZZ}{\mathbb Z}

\newcommand{\qed}{\hfill $\square$ \bigskip}

\begin{document}

\title{On $S$-packing total colorings}

\author{
Jasmina Ferme$^{a, b}$, Jaka Hed\v zet$^{b,c}$, Petra Melicharov\' a$^{d}$, Da\v sa Mesari\v c \v Stesl$^{e}$\\}

\date{}

\maketitle

\begin{center}
$^a$ Faculty of Education, University of Maribor, Maribor, Slovenia\\
\medskip

$^b$ Faculty of Natural Sciences and Mathematics, University of Maribor, Maribor, Slovenia\\
\medskip

$^c$ Institute of Mathematics, Physics and Mechanics, Ljubljana, Ljubljana, Slovenia\\
\medskip

$^d$ Faculty of Applied Sciences, University of West Bohemia, Plze\v n, Czech Republic
\medskip

$^e$ Faculty of Computer and Information Science, University of Ljubljana, Ljubljana, Slovenia\\
\medskip

\end{center}

\begin{abstract}
In this paper, we generalize the concept of packing total coloring by introducing a new concept called the \textit{$S$-packing total coloring}. For a graph $G$ and a non-decreasing sequence $S=(a_1,a_2,\ldots)$ of positive integers, an $S$-packing total coloring of $G$ is a mapping
$c: V(G)\cup E(G)\rightarrow \{1,2,\ldots\}$ 
such that for any two distinct elements $A,B\in V(G)\cup E(G)$ with $c(A)=c(B)=i$, the distance between $A$ and $B$ is at least $a_i+1$.
The smallest integer \(k\) such that \(G\) admits an \(S\)-packing total coloring using \(k\) colors is called the \textit{\(S\)-packing total chromatic number} of \(G\), denoted by \(\chi_S^{''}(G)\).
For any sequence \(S\), we establish general lower and upper bounds for \(\chi_S^{''}(G)\), and characterize all graphs \(G\) with \(\chi_S^{''}(G)\in\{1,2,3\}\). Furthermore, we investigate \(S\)-packing total chromatic numbers of complete bipartite graphs, as well as infinite and finite paths and cycles.

\end{abstract}

\noindent {\bf Key words: packing coloring, $S$-packing coloring, packing total coloring, $S$-packing total coloring. } 

\medskip\noindent
{\bf AMS Subj.\ Class: 05C15, 05C69, 05C70, 05C12}

\section{Introduction}

In 2008, Goddard and co-authors~\cite{goddard-2008} introduced the concept of an $S$-packing coloring, defined as follows. Given a graph $G$ and a non-decreasing sequence $S=(a_1,a_2,\ldots, a_k)$ of positive integers, an {\em $S$-packing $k$-coloring} of $G$ is a mapping $c:V(G) \rightarrow \{1,2,\ldots,k\}$ with the property that for any distinct vertices $u,v \in V(G)$ with $c(u)=c(v)=i$, the distance between $u$ and $v$ in $G$ is greater than $a_i$ (note that the distance between two vertices refers to the usual shortest-path distance). 
The smallest integer $k$ such that $G$ admits an $S$-packing $k$-coloring is called the \textit{$S$-packing chromatic number} of $G$ and is denoted by $\chi_S(G)$. In particular, when $S=(1, 2, 3, \ldots)$, an $S$-packing coloring is simply called a \textit{packing coloring} of $G$, and the corresponding $S$-packing chromatic number $\chi_S(G)$ is referred to as the \textit{packing chromatic number} of $G$, denoted by $\chi_\rho(G)$. 
Note that $S$-packing colorings, including packing colorings, have attracted significant interest from many researchers and, consequently, have been studied in a number of papers (e.g., a survey paper from 2020~\cite{survey} and more recent works~\cite{bresar2025, dliou, ferme-stesl, harith, hatting, peterin, junosza, mortada1, mortada2}).

Based on $S$-packing coloring, which refers to the coloring of the vertices of a given graph, the analogous concept for edge-coloring was developed, called $S$-packing edge-coloring. More precisely, for a non-decreasing sequence $S=(a_1,a_2,\ldots, a_k)$ of positive integers, an \textit{$S$-packing $k$-edge-coloring} of $G$ is a mapping $c':E(G) \rightarrow \{1,2,\ldots,k\}$ with the property that for any distinct edges $e,f \in E(G)$ with $c'(e)=c'(f)=i$, the distance between $e$ and $f$ in $G$ is at least $a_i+1$. Here, the distance between edges $e=ab$ and $f=cd$ is defined as $d(e,f)=\min \{d(a,c), d(a,d), d(b,c), d(b,d)\}+1$, which coincides with the distance between the corresponding vertices of $e$ and $f$ in the line graph of $G$.
Next, the smallest integer $k$ such that $G$ admits an $S$-packing $k$-edge-coloring is called the \textit{$S$-packing chromatic index} of $G$ and is denoted by $\chi_S'(G)$. In the special case when $S=(1,2,3, \ldots)$, an $S$-packing edge-coloring is called a \textit{packing edge-coloring}, and the corresponding invariant is the \textit{packing chromatic index} of $G$, denoted by $\chi_\rho'(G)$. Note that the concept of $S$-packing edge-coloring, including packing edge-coloring, is relatively new, but it has already been studied in several papers (see, e.g.,~\cite{gt-2019, hocquard, liu, yang}).

Recently, Ferme and \v Stesl~\cite{ferme-stesl} introduced a concept, called packing total coloring, which comprises both, packing coloring (of vertices) and packing edge-coloring. Indeed, a \textit{packing $k$-total coloring} of a given graph $G$ is defined as a mapping $c: V(G) \cup E(G) \rightarrow \{1,2, \ldots, k\}$ which posseses the property that for any $i \in \{1,2, \ldots, k\}$ and any two distinct elements $A, B \in V(G) \cup E(G)$ the following implication holds: if $c(A)=c(B)=i$, then $d(A, B) > i$. Note that $d(A,B)$ is the a) length of a shortest-path between $A$ and $B$ if $A, B \in V(G)$; b) $\min \{d(a,d), d(a,c),d(b,c), d(b,d)\}+1$ if $A, B \in E(G)$ and $A=ab$, $B=cd$; c) $ \min \{d(a,B), d(b,B)\}+1$ if $A=ab \in E(G)$ and $B \in V(G)$; d) $ \min \{d(A,c), d(A,d)\}+1$ if $A \in V(G)$ and $B=cd \in E(G)$. That is, the distance between $A, B \in V(G) \cup E(G)$ coincides with the distance between the corresponding vertices in the \textit{total graph} of a graph $G$, denoted by $T(G)$, with $V(T(G))=V(G) \cup E(G)$ and $E(T(G))=E(G) \cup \{AB;~A,B \in E(G), A,B~\textrm{are incident in} ~G\} \cup  \{aB;~a \in V(G),B \in E(G), a~\textrm{is an endpoint of}~B~\textrm{in}~G\}$. This implies that the packing total coloring of a given graph $G$ is equivalent to the packing coloring of the vertices of the total graph of $G$. Further, the \textit{packing total chromatic number} of $G$, denoted by $\chi_\rho''(G)$, is the smallest integer $k$ for which there exists a packing $k$-total coloring of $G$.

In this paper, we present a generalization of packing total coloring, namely the concept of \textit{$S$-packing total coloring}. Specifically, instead of considering the sequence $S=(1,2,3, \ldots)$ used for a packing total coloring, we consider an arbitrary non-decreasing sequence $S=(a_1,a_2,\ldots)$ of positive integers. 

For any graph $G$ and any non-decreasing sequence $S=(a_1,a_2,\ldots)$ of positive integers, an $S$-packing total coloring 
of a $G$ is a mapping $c: V(G) \cup E(G) \rightarrow \{1,2, \ldots\}$ such that for any $i \in \{1,2, \ldots\}$ and any two distinct elements $A, B \in V(G) \cup E(G)$ the following implication holds: if $c(A)=c(B)=i$, then $d(A, B) > a_i$, where $d(A,B)$ means a) the usual shortest-path distance between $A$ and $B$ if $A, B \in V(G)$; b) the $\min \{d(a,d), d(a,c),d(b,c), d(b,d)\}+1$ if $A, B \in E(G)$ and $A=ab$, $B=cd$; c) the $ \min \{d(a,B), d(b,B)\}+1$ if $A=ab \in E(G)$ and $B \in V(G)$; d) the value of $ \min \{d(A,c), d(A,d)\}+1$ if $A \in V(G)$ and $B=cd \in E(G)$. Recall that the distance between $A, B \in V(G) \cup E(G)$ coincides with the distance between the corresponding vertices in $T(G)$. Consequently, the concept of $S$-packing total coloring of a given graph $G$ is equivalent to an $S$-packing coloring (of the vertices) of the total graph of a graph $G$. 
Next, the smallest integer $k$ such that there exists an $S$-packing total coloring of $G$ using $k$ colors, is called the \textit{$S$-packing total chromatic number} of $G$ and is denoted by $\chi_S^{''}(G)$. Note that this notation coincides with the notations for $S$-packing coloring (by using $\chi_S$) and total coloring (with $''$).

As already mentioned, this paper introduces a new concept called \textit{$S$-packing total coloring}. We begin by defining the concept and the related terminology, and then provide several lower and upper bounds for the $S$-packing total chromatic number of graphs. Furthermore, we characterize all graphs $G$ with $\chi_S^{''}(G) \in \{1,2,3\}$ for an arbitrary sequence $S$.
In Section 4, we investigate $S$-packing total colorings of complete bipartite graphs. We then turn our attention to the infinite path, presenting its $S$-packing total chromatic numbers for all sequences $S$ consisting only of integers $1$ and $2$. In addition, we identify all sequences $S$ for which the $S$-packing total chromatic number of the infinite path is at most $5$. Section 6 is devoted to finite paths and cycles, where we study their $S$-packing total chromatic numbers. Finally, we conclude with several questions and open problems that may serve as directions for future research.


\section{Notations and preliminaries}
In this paper, we consider only simple, undirected graphs. For any such graph $G$ we denote its vertex set by $V(G)$ and its edge set by $E(G)$. 
Let $u$ and $v$ be two distinct vertices from $V(G)$. We say that $u$ is a \textit{neighbor} of $v$ ($v$ is a \emph{neighbor} of $u$, respectively) if $uv \in E(G)$. The set of all neighbors of $u$ is called the \emph{(open) neighborhood} of $u$ and is denoted by $N_G(u)$. The cardinality of $N_G(u)$ is the \emph{degree} of $u$ and is denoted by $deg_G(u)$. In the case when $deg_G(u) = 1$, $u$ is a \textit{leaf}. 
Note that the subscript in the above notations may be omitted if the graph $G$ is clear from context. Further, the \emph{maximum degree} of $G$, denoted by $\Delta(G)$, is defined as $\Delta(G) =\max\{deg(v);~v \in V(G)\}$. 
%
Next, any two distinct edges $e_1, e_2\in E(G)$ are \emph{incident} if they have a common endpoint. The subset of edges $M\subseteq E(G)$ is called a \emph{matching} in a graph $G$ if no two edges of $M$ are incident. 

A graph $H$ is a subgraph of $G$ if $V(H) \subseteq V(G)$ and $E(H) \subseteq E(G)$. 

Next, a \emph{bipartite graph} is a graph $G$ whose vertex set can be partitioned into two disjoint, non-empty sets $V_1$ and $V_2$ such that no two vertices within $V_1$ are adjacent and no two vertices within $V_2$ are adjacent. Equivalently, all edges of $G$ have one endpoint in $V_1$ and the other in $V_2$. Note that the mentioned partition is called a \emph{bi-partition}.
 A bipartite graph $G$ with bi-partition $(V_1,V_2)$ is a \emph{complete bipartite graph} if each vertex from $V_1$ is adjacent to all vertices from $V_2$. A complete bipartite graph whose partite sets satisfy $|V_1|=m$ and $|V_2|=n$ is denoted by $K_{m,n}$. Additionally, a complete bipartite graph $K_{1,n}$ consisting of a central vertex connected to $n$ leaves is called a \emph{star}. 

Let $u,v \in V(G)$ and $ab,cd \in E(G)$. The \emph{distance between the vertices $u$ and $v$}, denoted by $d_G(u,v)$ is defined as the length of a shortest path between $u$ and $v$ in $G$. Next, the \emph{distance between the vertex $u$ and the edge $ab$} is $d(u,ab)=min\{d(u,a), d(u,b)\}+1$. Finally, the \emph{distance between two edges} in $G$, $ab$ and $cd$, is defined as follows: $d(ab,cd)=\min\{d(a,c), d(a,d), d(b,c), d(b,d)\}+1$.

 For a given a graph $G$, the \emph{total graph} of $G$, denoted by $T(G)$, is the graph with the vertex set $V(T(G)) = V(G) \cup E(G)$ and property that two vertices in $T(G)$ are adjacent if and only if they correspond to adjacent vertices in $G$, to incident edges in $G$, or when one corresponds to a vertex $v\in V(G)$ and the other to an edge $e\in E(G)$ with endpoint $v$.

In this paper, we introduce the concept of an $S$-packing total coloring. For a given non-decreasing sequence $S=(a_1, a_2, a_3, \ldots)$ of positive integers, it is defined as a mapping $c: V(G) \cup E(G) \rightarrow \{1,2, \ldots\}$ with the property that for any $i \in \{1,2, \ldots\}$ and any two distinct elements $A, B \in V(G) \cup E(G)$ the following implication holds: if $c(A)=c(B)=i$, then the distance between $A$ and $B$ is at least $a_i+1$. If for a finite sequence $S=(a_1, a_2, a_3, \ldots, a_k)$ there exists and $S$-packing total coloring of a given graph $G$, then we say that $G$ is $S$-packing total colorable graph. 
Next, the smallest integer $k$ such that there exists an $S$-packing total coloring of $G$ using colors from $\{1,2,\ldots,k\}$, is the $S$-packing total chromatic number of $G$, denoted by $\chi_S^{''}(G)$. For a given $S$-packing total coloring $c$ of $G$, we will denote the number of elements of $G$ to which $c$ assigns a color $i$, by $|c^{-1}(i)|$.

Using the definition of a total graph of a given graph $G$ and the fact that for any $A,B \in V(G) \cup E(G)$, $d_G(A,B)$ is the same as the distance between the vertices $A$ and $B$ in $T(G)$, we derive that a concept of an $S$-packing total coloring of a given graph $G$ coincides with the concept of $S$-packing coloring (of the vertices) of $T(G)$. Consequently, $\chi_S^{''}(G)=\chi_S(T(G))$ for any graph $G$.

In the continuation of this paper, each sequence $S$ will be a non-decreasing sequence of positive integers. When only the first few terms of a sequence $S$ are specified and the notation is followed by three dots, the dots represent an arbitrary continuation of the sequence, rather than a repetition of the last listed integer. For example, the notation $S=(1,2,2,\ldots)$ denotes a sequence beginning with $1,2,2$, followed by arbitrary additional integers, which are not necessarily equal to $2$. Additionally, we use the notation $a^r$ in sequences $S$, where $a$ and $r$ are positive integers, to denote $r$ consecutive occurrences of the value $a$ in $S$. Therefore, $a^r$ does not represent a power of $a$, but, for instance, $S=(2^5,4)$ stands for $S=(2,2,2,2,2,4)$.

In some cases, an $S$-packing total coloring of a graph will be described by a \emph{(periodic) color pattern} of the form $[c_1,\dots,c_d]$ (written in square brackets), where the colors $c_n$, for $n \in \{1,\dots,d\}$, are not necessarily pairwise distinct. The colors from the pattern are assigned repeatedly to consecutive elements of the graph, which are specified in the corresponding section. More precisely, applying a color pattern $[c_1,\dots,c_d]$ means that we color the consecutive vertices/edges of a graph one after another using colors $c_1,c_2,\dots,c_d, c_1,c_2,\dots,c_d,c_1,c_2,\dots,c_d, \ldots$

\section{Basic observations and propositions}

In the first part of this section, we provide some bounds on $S$-packing total chromatic numbers of graphs. Further, we characterize all graphs $G$ with $\chi_S^{''}(G)\in \{1,2,3\}$ for an arbitrary sequence $S$. 

For any sequence $S=(a_1,a_2,a_3,\ldots)$, an $S$-packing coloring of a graph assigns colors to its vertices so that any two distinct vertices receiving color $i$ are at distance greater than $a_i$. Likewise, an $S$-packing edge-coloring of a graph $G$ assigns colors to the edges of $G$ so that any two distinct edges colored by $i$ are at distance greater than $a_i$.
Since an $S$-packing total coloring must simultaneously satisfy these vertex and edge conditions (in addition to the constraints involving distances between vertices and edges), the following observation is immediate.

\begin{observation}
    Let $G$ be a graph. Then, for any $S$, $\chi_S^{''}(G)\geq max \{\chi_{S}(G), \chi^{'}_{S}(G)\}.$
\end{observation}

Moreover, it is known that $\chi(K_n)=\chi^{'}(K_n)=\chi^{''}(K_n)$ if $n \geq 3$ is odd~\cite{behzad}, which means that in the case when $S=(1,1,1, \ldots)$ there exists a graph with $\chi_S(G)=\chi_S^{'}(G)=\chi_S^{''}(G)$. Hence the above written bound is sharp

Additionally, note that the difference between $\chi_S^{''}(G)$ and $ \chi_{S}(G)$ can be arbitrarily large. Indeed, Ferme and Mesari\v c \v Stesl~\cite{ferme-stesl} proved that $\chi_{S}^{''}(K_{1,n}) - \chi_{S}(K_{1,n})=n$ for any $n$ in the case when $S=(1,2,3, \ldots)$. 

Note that the previous observation provides a lower bound for $\chi_S^{''}(G)$. The following observation gives an additional lower bound as well as an upper bound for $\chi_S^{''}(G)$.

\begin{observation} \label{obs3}
    Let $S=(a_1,a_2,\dots)$ and let $G$ be a finite graph with $|V(G)|=n$ and $|E(G)|=m$. Then $1 + \Delta(G) \leq \chi_S^{''}(G) \leq n+m$. Moreover, $\chi_S^{''}(G) = n+m$ if and only if $G$ is connected and $a_1 \geq \max_{A,B \in V(G) \cup E(G)} d(A,B)$.
\end{observation}

\begin{proof}
       Let $x$ be a vertex of $G$ with $\deg(x)=\Delta(G)$. Since any two distinct elements from $\{x\} \cup \{xv;~xv \in E(G)\}$ are pairwise at distance $1$, any $S$-packing total coloring $c$ of $G$ assigns to these elements pairwise distinct colors. This means that $c$ uses at least $|\{x\} \cup \{xv;~xv \in E(G)\}|=1+\Delta(G)$ colors, hence $\chi_S^{''}(G) \geq 1+\Delta(G)$.
       
    Next, it is clear that $\chi_S^{''}(G) \leq n+m$. Therefore, it remains to show that $\chi_S^{''}(G) = n+m$ if and only if $G$ is connected and $a_1 \geq \max_{A,B \in V(G) \cup E(G)} d(A,B)$. Firstly, if $a_1 < \max_{A,B \in V(G) \cup E(G)} d(A,B)$, then at least two elements of $V(G) \cup E(G)$ can receive color $1$. Therefore $\chi_S^{''}(G) < n+m$.
    On the contrary, if $a_1 \geq \max_{A,B \in V(G) \cup E(G)} d(A,B)$, then $a_i \geq \max_{A,B \in V(G) \cup E(G)} d(A,B)$ for all $i$ and each corresponding color $i$ can be used at most once, hence $\chi_S^{''}(G) = n+m$.
\qed
\end{proof}








Let $S_1=(a_1,a_2,\dots)$ and $S_2=(b_1,b_2,\dots)$ be two non-decreasing sequences of positive integers such that $b_i \leq a_i$ for all $i$. Then, it is clear that each $S_1$-packing total coloring of a graph $G$ is also an $S_2$-packing total coloring of this graph. Thus, the following observation holds.

\begin{observation} \label{obs1}
    Let $S_1=(a_1,a_2,\dots)$ and $S_2=(b_1,b_2,\dots)$ such that $b_i \leq a_i$ for all $i \in \{ 1,\dots,k\}$. Then, for any graph $G$, $\chi_{S_1}^{''}(G)=k$ implies that $\chi_{S_2}^{''}(G) \leq k$.
\end{observation}

    

Next, for a given sequence $S$, we consider an $S$-packing total coloring of a subgraph $H$ of a given graph $G$. Note that the distance between any two objects in $H$ is greater or equal to the distance between them in $G$.
Hence, any $S$-packing total coloring of $G$ is also an $S$-packing total coloring of $H$, which gives us the next observation.

\begin{observation} \label{obs2}
    Let $S=(a_1,a_2,\dots)$ and let $H$ be a subgraph of a given graph $G$. Then, $\chi_S^{''}(H) \leq \chi_S^{''}(G)$.
\end{observation}



Next, we continue with a characterization of the graphs with $S$-packing total chromatic number at most $3$.

\begin{proposition}
\label{prop_1_3}
    Let $S=(a_1,a_2,\dots)$. Then $\chi_S^{''}(G) = 1$ if and only if $G$ has no edges. Moreover, there is no graph $G$ with $\chi_S^{''}(G) = 2$.
\end{proposition}

\begin{proof}
    Let $S=(a_1,a_2,\dots)$. If $G$ contains at least one edge, $e=xy$, then the objects $e,x$ and $y$ must receive three pairwise distinct colors by any $S$-packing total coloring $c$ of $G$ and hence, $\chi_S^{''}(G) \geq 3$. Therefore, $\chi_S^{''}(G) \in \{1,2\}$ holds only for graphs $G$ without edges. Since in such graphs all vertices are non-adjacent, they can receive the same color by $c$. Consequently, for any graph $G$ with no edges (and with at least one vertex) $\chi_S^{''}(G) = 1$, and for any other graph $G$, $\chi_S^{''}(G) \geq 3$.
\qed 
\end{proof}

\begin{proposition} \label{prp: chi3}
    Let $S=(a_1,a_2,a_3, \dots)$ and let $G$ be a connected graph. Then the following holds.
    \begin{enumerate} [$(1)$]
        \item If $a_3=1$, then $\chi_S^{''}(G) = 3$ if and only if $G$ is a path $P_n$, $n \geq 2$, or $G$ is a cycle $C_n$ such that $n \equiv 0 \ (\mathrm{mod} \ 3)$.
       \item If $a_3 \geq 2$ and $a_2=1$, then $\chi_S^{''}(G) = 3$ if and only if $G$ is either a path $P_2$ or a path $P_3$.
          \item If $a_3 \geq 2$ and $a_2 \geq 2$, then $\chi_S^{''}(G) = 3$ if and only if $G$ is a path $P_2$.
    \end{enumerate}
\end{proposition}

\begin{proof}
        Let $S=(a_1,a_2, a_3,\dots)$ be a non-decreasing sequence of positive integers and let $G$ be a connected graph with $\chi_S^{''}(G) = 3$. Based on Observation \ref{obs3}, we know that $3 = \chi_S^{''}(G) \geq \Delta(G)+1$, which implies that $\Delta(G) \leq 2$. Since $G$ is a connected graph, it is a path or a cycle. Moreover, from the proof of Prop.~\ref{prop_1_3} it follows that $G$ has at least one edge, which means, that it is a path with at least $2$ vertices or a cycle.
    
    First, consider the case when $a_3=1$, and consequently, $a_1=a_2=1$.
    If $G$ is a path $P_n$, where $n \geq 2$, then there exists a $(1^3)$-packing total coloring $c$ given by a periodic pattern $[1,2,3]$, hence $\chi_S^{''}(P_n) = 3$ for any $n \geq 3$.
    Moreover, this coloring is unique up to the permutation of colors. Indeed, for an edge $e=xy$ and its endpoints $x$ and $y$ we need three different colors, without loss of generality we can have $c(x)=1, c(e)=2$ and $c(y)=3$, which determines the $S$-packing total coloring of all other uncolored objects.
    Hence, if $G$ is a cycle $C_n$, where $n \geq 3$, then it is $S$-packing total colorable with $3$ colors if and only if $|V(C_n) \cup E(C_n)| = 2n$ is divisible by $3$. Note that this condition means that $n \equiv 0 \ (\mathrm{mod} \ 3)$. Clearly, such cycle can be colored by the periodic pattern $[1,2,3]$.

    Next, assume that $a_3 \geq 2$ and $a_2=1$. If $G$ contains (at least) two vertices of degree $2$ (this means that $G$ is a path with at least $4$ vertices or a cycle), it is easy to see that it cannot be $S$-packing total colorable using only three colors. Hence, $G$ is a path of order $3$ or $2$. Clearly, such path can be colored by the periodic pattern $[1,2,3]$.

    Now, let $a_3 \geq 2$ and $a_2 \geq 2$. If $G$ contains a vertex with degree $2$, it cannot be $S$-packing total colorable using only three colors, which implies that $G$ is isomorphic to $P_2$. Note that this holds if $a_1 \geq 2$ and also in the case when $a_1=1$. This concludes the proof. 
    
\qed 
\end{proof}

\section{$S$-packing total colorings of complete bipartite graphs}

We next focus on complete bipartite graphs and determine their S-packing total chromatic numbers for all sequences S consisting only of the integers $1$ and $2$.

In the case of classical coloring ($S=(1,1,1,1,\ldots, 1)$), the values for the $S$-packing total chromatic numbers of complete bipartite graphs are known. Namely, if $m=n$, then $\chi_S^{''}(K_{m,n})=n+2$ and otherwise, $\chi_S^{''}(K_{m,n})=n+1$~\cite{behzad}. Note that this result also implies that if $m=n$, then $\chi_S^{''}(K_{m,n})=n+2$, and otherwise, $\chi_S^{''}(K_{m,n})=n+1$, when $S=(1^{n+2},\ldots)$, that is, when at least the first $n+2$ elements of $S$ are equal to $1$. Consequently, in the following theorem, we consider only sequences $S$ that do not satisfy this condition.

\begin{theorem} \label{prp: complete}
    Let $S=(a_1,a_2,\dots)$ be a non-decreasing sequence of positive integers and let $K_{m,n}$, $1 \leq m \leq n$, be a complete bipartite graph. If $m<n$, let $D=\frac{n}{n-m}$. Then, the following holds.
    \begin{enumerate} [$(1)$]
        \item If $a_1=a_2=\ldots=a_{p}=1$, $a_{p+1}\ge 2$, where $1 \leq p \leq n+1$, \\ then
        $$ \chi_S^{''}(K_{m,n})=\left\{
	\begin{array}{ll}
        2m+m^2-p(m-1) ;\ m=n\\ 
		m+n+mn-p(n-1);\ m<n\ \textrm{and}\ p \leq D,\\
		m+mn- p(m-1);\ m<n\ \textrm{and}\ p > D.\ 
	\end{array}\right.
    $$

        \item If $a_1 \geq 2$, $\chi_S^{''}(K_{m,n}) = m+n+mn$.
    \end{enumerate}
\end{theorem}

\begin{proof}
    Let $S=(a_1,a_2,\dots)$ be a non-decreasing sequence of positive integers. Further, let $G=K_{m,n}$ be an arbitrary complete bipartite graph with bi-partition $(A,B)$, where $|A|=m$, $|B|=n$ and $1 \leq m \leq n$. Clearly, $|V(G)|=m+n$ and $|E(G)|=m\cdot n.$ We also know that for any two elements $a,b \in V(G) \cup E(G)$, $d(a,b) \leq 2$.  Thus, point (2) immediately follows.

    Now let $a_1=a_2=\ldots=a_p=1$ and $a_{p+1}\geq 2$, where $1 \leq p \leq n+1$. Denote by $c$ any $S$-packing total coloring of $G$. We notice that $|c^{-1}(j)| \leq 1$ for any $j \geq p+1$. 
    %
    Therefore, in order to determine the lower bound for $\chi_S^{''}(G)$, we consider the maximum number of elements from $V(G) \cup E(G)$ to which $c$ assigns colors $1, 2, \ldots, p$. %
    Let $i \in \{1,2,\ldots,p\}$ be an arbitrary integer. 
    If there does not exist a vertex from $V(G)$ colored with $i$ by $c$, then color $i$ can be assigned to at most $m$ pairwise non-incident edges of $G$. Otherwise, let $C_i=\{x \in V(G) \, | \, c(x)=i\} \neq \emptyset$. Clearly, $C_i \subseteq A$ or $C_i \subseteq B$.\\
    In the first case, $|A\setminus C_i|<|B|$, which means that at most $|A\setminus C_i|$ edges of $G$ receive a color $i$ by $c$. Therefore, $c$ assigns a color $i$ to at most $|C_i|+|A\setminus C_i|=|A|=m$ elements of $G$. \\
    Now, assume that $C_i \subseteq B$. In this case, $c$ assigns a color $i$ to at most $min\{|A|, |B \setminus C_i|\}$ edges of $G$, since any edge of $G$ colored by $i$  has one endpoint in $A$ and the other in $B \setminus C_i$. 
    We observe the following. If $|C_i| < |B|-|A|$, then $min\{|A|, |B \setminus C_i|\}=|A|$, which means that $c$ assigns color $i$ to at most $|C_i|+|A| < |B|-|A|+|A|=n$ elements of $G$. Otherwise, $|C_i| \geq |B|-|A|$ and $c$ assigns color $i$ to at most $|C_i|+|B \setminus C_i|=|B|=n$ elements of $G$.
    
    %
Based on these findings, we conclude that $c$ assigns each color $i \in \{1,2,\ldots,p\}$ to at most $n$ elements from $V(G) \cup E(G)$. Moreover, if $n>m$, then this value can be achieved only when $C_i \neq \emptyset$, $C_i \subseteq B$ and $|C_i| \geq |B|-|A|$. If $n=m$, then each color $i \in \{1,2,\ldots,p\}$ can be assigned to at most $n=m$ elements and this bound can be achieved in each of the cases; either $C_i = \emptyset$ or $C_i \neq \emptyset$ (and $C_i \subseteq A$ or $C_i \subseteq B$).

In continuation of the proof, we distinguish two cases. 

First, consider the case when $n=m$. This means that each color $i \in \{1,2,\ldots,p\}$ can be assigned to at most $n=m$ elements of $G$. Since $|A|=|B|=m$ and $|E(G)|=m^2$, we can have at most $\frac{2m+m^2}{m}=2+m$ distinct colors from $\{1,2,\ldots,p\}$ that achieve this bound. \\
%
%
Note that since $p\leq n+1=m+1$, we have $p<m+2$. Recall that each of the colors from $\{1,2,\ldots, p\}$ can be assigned to at most $m$ elements of $G$, while any other color to at most one element of $G$. Consequently, $\chi_S^{''}(G) \geq 2m+m^2-pm+p$. 
Moreover, we can form an $S$-packing total coloring of $G$ using $2m+m^2-pm+p$ colors. \\
First, let $p=1$. In this case color all vertices in $A$ using a color $1$. Then, color the remaining $m^2+m$ elements of $G$ using pairwise distinct colors from $\{2,3,\ldots,m^2+m+1\}$. Since in this way an $S$-packing total coloring of $G$ is formed, we conclude that $\chi_S^{''}(G)=m^2+m+1$. \\
If $p=2$, then color all vertices from $A$ with a color $1$ and all vertices from $B$ with a color $2$. Then, color the remaining $m^2$ elements of $G$ using pairwise distinct colors from $\{3,\ldots,m^2+2\}$. It is clear that the described coloring is an $S$-packing total coloring of $G$, hence $\chi_S^{''}(G)=m^2+2$ and our claim holds. \\
Finally, assume that $p \geq 3$. Since $1 \leq p-2 < m$, $E(G)$ contains disjoint sets $E_1, E_2, \ldots, E_{p-2}$ such that for each $i \in \{1,2,\ldots,p-2\}$, $E_i$ is a perfect matching consisting of $m$ edges. First, for each $i \in \{1,2,\ldots,p-2\}$, color all edges from $E_i$ with a color $i$. Then, color all vertices in $A$ using a color $p-1$ and all vertices from $B$ with a color $p$. Finally, color the remaining elements in $G$ with new, distinct colors from $\{p+1, p+2, \ldots, 2m+m^2-pm+p\}$. In this way, an $S$-packing total coloring of $G$ using $2m+m^2-pm+p$ colors is obtained, thus $\chi_S^{''}(G) = 2m+m^2-pm+p=2m+m^2-p(m-1)$.

  
%
   Now, assume that $n \neq m$ and consequently, $|B|-|A| \neq 0$. Let $D=\frac{|B|}{|B|-|A|}=\frac{n}{n-m}$. 
   Recall that $c$ assigns each color $i \in \{1,2,\ldots,p\}$ to at most $n$ elements of $G$ and this value can occur only when $C_i \neq \emptyset$ and $C_i \subseteq B$. Moreover, the mentioned bound can be achieved only when $|C_i|\geq|B|-|A|$. 
   Note that there is at most $ \lfloor D \rfloor$ colors $i$ for which we can form such $C_i$ that $|C_i|\geq|B|-|A|$. Without loss of generality assume that these colors are $1,2, \ldots, \lfloor D \rfloor $ if $p \geq \lfloor D \rfloor $, and $1,2,\ldots, p$ if $p \leq \lfloor D \rfloor $.
 As mentioned above, $c$ assigns each of these colors to at most $n$ elements of $G$. More precisely, each of these colors is assigned to $|C_i|$ vertices of $B$ and $|B \setminus C_i|$ edges of $G$.

 In order to determine $\chi_S^{''}(G)$, we now distinguish two cases.

 \textbf{Case 1. $p \leq D$}.\\
In this case, each of the colors from $\{1,2,\ldots,p\}$ can be assigned to at most $n$ elements of $G$ and any other color to at most one element of $G$. Hence, $\chi_S^{''}(G) \geq m+n+mn- pn+p$. \\
In order to prove that $\chi_S^{''}(G)=m+n+mn- pn+p$, we construct an $S$-packing total coloring $c$ of $G$ using $m+n+mn-pn+p$ colors. 
Let $|B|-|A|=k$. Since $p \leq D=\frac{n}{k}$, we have $pk \leq n=k+m$ and thus, $p \leq 1+\frac{m}{k}$. Consequently, $p \leq 1+m$. \\
%
Next, since $p \leq D$, there exist $p$ pairwise disjoint subsets of $B$, namely $C_1, C_2, \ldots, C_p$ such that $|C_i|\geq n-m$ for each $i \in \{1,2,\ldots, p\}$. Let $c$ assign color $i$ to all vertices from $C_i$ for any $i \in \{1,2,\ldots, p\}$. Further, since there is $|B \setminus C_i|$ pairwise non-incident edges with one endpoint in $A$ and another in $B \setminus C_i$ ($|A| \geq |B \setminus C_i|$), color all of them with a color $i$ for each $i \in \{1,2,\ldots, p\}$. Note that there exist disjoint sets of such edges for all colors from $\{1,2,\ldots, p\}$. Indeed, applying the presented coloring, each vertex $a \in A$ is an endpoint of at most $p$ colored edges, where $p \leq m+1 \leq n=deg_G(a)$. Similarly, after applying the presented coloring, each vertex $b \in B$ is an endpoint of at most $p-1$ colored edges, where $p-1 \leq m =deg_G(b)$. Hence, such edges exist. 
%
%
%
Note that in this way, each of the colors from $\{1,2,\ldots, p\}$ is assigned to exactly $n$ elements of $G$. Further, let $c$ assign new, distinct colors to the remaining elements of $G$ (note that any of these new colors can be assigned only to one element of $G$). This means that $c$ uses at most $m+n+mn-pn+p$ colors, which confirms that $\chi_S^{''}(G) = m+n+mn- pn+p$. 
%

%
%
\textbf{Case 2. $p > D$}.\\
Note that if $p=n+1$, then it is already known that $\chi_S^{''}(G) =n+1$ and our claim holds. Therefore, in continuation of the proof we assume that $p \leq n$. 
Recall that for any $i \in \{1,2,\ldots, p\}$, $c$ assigns a color $i$ to at most $n$ elements of $G$. More precisely, $|c^{-1}(i)|\leq m+|C_i \cap B|$. Indeed, if $C_i \subseteq A$, then at most $m$ elements of $G$ receive a color $i$, and otherwise to at most $|C_i|$ vertices in $B$ and $m$ edges of $G$ is assigned a color $i$. Since $C_i \cap C_j = \emptyset$, for any distinct sets $C_i, C_j \subseteq B$, we also have $\sum_{i=1}^p|C_i \cap B| \leq n$. Hence, $\sum_{i=1}^p|c^{-1}(i)| \leq pm+n$. Additionally, since for each color $i$ we know that $|c^{-1}(i)|\leq n$, we have $\sum_{i=1}^p|c^{-1}(i)| \leq \min\{pn,\ pm+n\}$. Further, since $p>D$, we know that $pn \geq pm+n$, thus $\sum_{i=1}^p|c^{-1}(i)| \leq pm+n$. Therefore, $c$ uses at least $m+n+mn-(pm+n)+p$ colors, which means that $\chi_S^{''}(G) \geq m+mn- pm+p$. \\
In order to prove that  $\chi_S^{''}(G) = m+mn- pm+p$, we form an $S$-packing total coloring $c$ of $G$ using $m+mn-pm+p$ colors. 
First, note that there exist $\lfloor D \rfloor$ pairwise disjoint subsets of $B$, namely $C_1, C_2, \ldots, C_{\lfloor D \rfloor}$ such that $|C_i|=n-m$ for each $i \in \{1,2,\ldots, \lfloor D \rfloor\}$. Since $p>D$, and hence $p\geq \lfloor D\rfloor+1$, let $C_{\lfloor D \rfloor+1}=B \setminus \bigcup_{i=1}^{\lfloor D \rfloor}C_i$. Further, let $c$ assign a color $i$ to all vertices from $C_i$ for any $i \in \{1,2,\ldots, \lfloor D \rfloor, \lfloor D \rfloor+1\}$.
Next, since for each $i \in \{1,2,\ldots, \lfloor D \rfloor\}$ there is $m$ pairwise non-incident edges with one endpoint in $A$ and another in $B \setminus C_i$ (recall that $|A|=|B \setminus C_i|=m$), color all of them with a color $i$. For the color $\lfloor D\rfloor+1$, observe that $\lfloor D\rfloor(n-m)\geq (D-1)(n-m)=D(n-m)-(n-m)=n-n+m=m$. Consequently, $|B\setminus C_{\lfloor D\rfloor+1}|=\lfloor D\rfloor(n-m)\geq m$, and thus there are $m$ pairwise non-incident edges with one endpoint in $A$ and another in $B\setminus C_{\lfloor D\rfloor+1}$. Color these edges with color $\lfloor D\rfloor+1$.
Note that there exist disjoint sets of such edges for all colors from $\{1,2,\ldots, \lfloor D \rfloor+1\}$, since we have already colored $(\lfloor D \rfloor +1)m$ edges, but $G$ contains $nm \geq pm \geq (D+1)m \geq (\lfloor D \rfloor+1)m$ edges. 
%
%
%
If $p \geq \lfloor D \rfloor\ + 2$, then let each of the colors  $\lfloor D \rfloor\ + 2, \ldots, p$ be assigned to pairwise non-incident $m$ edges of $G$. Note that in this way we color $pm$ edges and $|E(G)|=mn \geq pm$. Finally, let the remaining elements of $G$ receive pairwise distinct colors $p+1, \ldots$ The described coloring is an $S$- packing total coloring of $G$, which uses $m+mn-pm+p$ colors. 

    \qed
\end{proof} 

\section{$S$-packing total colorings of infinite path}

In this section, we consider $S$-packing total colorings of infinite path. First, recall that the distance graph $ G(\mathbb{Z}, \{1,2\})$, or shorter $G(1,2)$, is the infinite graph with $\mathbb{Z}=\{\ldots, -2, -1, 0, 1, 2, \ldots\}$ as the vertex set, while vertices $x$ and $y$ are adjacent if and only if $ |x-y| \in \{1,2\}$. We observe that the total graph of the infinite path, $T(P_\infty)$, is isomorphic to the distance graph $G(1,2)$, hence the following results from~\cite{ben, holub} holds for $\chi_S^{''}(P_\infty)$.

\begin{theorem}[\cite{ben, holub}]
\label{infinite_path_known}
\hfill
    \begin{enumerate}
         \item If $S=(1,1,1, \ldots)$, then $\chi_S^{''}(P_\infty)=3$. 
          \item If $S=(1,1,2,2, \ldots)$, then $\chi_S^{''}(P_\infty)=4$.
         \item If $S=(1,2,2,2,2, \ldots)$, then $\chi_S^{''}(P_\infty)=5$.
         \item If $S=(2,2,2,2,2, \ldots)$, then $\chi_S^{''}(P_\infty)=5$.
    \end{enumerate}
\end{theorem}

Additionally to the first claim of this theorem, it follows immediately from Prop.~\ref{prp: chi3} that the sequence $S=(1,1,1,\ldots)$ is the only sequence for which $\chi_S^{''}(P_\infty)=3$. In the continuation, we also prove that the sequence $S=(1,1,2,2,\ldots)$ is the only sequence for which $\chi_S^{''}(P_\infty)=4$.

In the proofs, we will adopt the following notation. For the infinite path $P_\infty$ let $V(P_\infty)=\{\ldots,-2,-1,0,1,2,\ldots\}$ and let $e_i$ denote the edge between vertices $i$ and $j$ for every $i,j \in V(P_\infty)$ and $i < j$. We will also use the term \emph{consecutive elements/objects} of $P_\infty$, meaning consecutive elements of $V(P_\infty) \cup E(P_\infty)$ with respect to the natural ordering:
$\ldots, -2, e_{-2}, -1, e_{-1}, 0, e_0, 1, e_1, 2, e_2, \ldots $

\begin{proposition} \label{prp: ZZsum}
    Let $S=(a_1,a_2,\dots)$. If $\chi_S^{''}(P_\infty) \leq k$, then $ \sum_{i=1}^k \frac{1}{2a_i +1} \geq 1$.
\end{proposition}

\begin{proof}
    Let $S=(a_1,a_2,\dots)$ and consider any $S$-packing total coloring $c$ of $P_\infty$ using $k$ colors. 
    Further, let $c(x)=i$ for some vertex $x \in V(P_\infty)$ and some $i\in\{1,\ldots,k\}$.
    This implies that none of the vertices $x+1,\ldots,x+a_i$ can be assigned color $i$, since each of them is at distance at most $a_i$ from $x$. Moreover, none of the edges $e_x,e_{x+1},\ldots,e_{x+a_i-1}$ can receive a color $i$ by $c$, as each of them is also at distance at most $a_i$ from $x$.
    Therefore, among any $2a_i+1$ consecutive elements of $P_\infty$, at most one can receive color $i$ by $c$. Indeed, within any block consisting of $a_i+1$ consecutive vertices of $P_\infty$ together with the $a_i$ edges between them, to at most one element can be assigned color $i$ by $c$. This means that any color $i$ can be assigned by $c$ to at most $\frac{1}{2a_i +1}$ of all objects of $P_\infty$.
%
    Therefore, if there exists an $S$-packing total coloring of $P_\infty$ using $k$ colors, then each object of $P_\infty$ receives one of $k$ possible colors, hence $ \sum_{i=1}^k \frac{1}{2a_i +1} \geq 1$.
\qed    
\end{proof}

\begin{theorem} \label{prp: ZZ4}
    Let $S=(a_1,a_2, a_3, a_4, \dots)$. Then $\chi_S^{''}(P_\infty)=4$ if and only if $a_1=a_2=1$ and $a_3=a_4=2$.
\end{theorem}

\begin{proof}
    Let $S = (a_1,a_2,a_3,a_4)$ be an arbitrary non-decreasing sequence of positive integers such that $\chi_S^{''}(P_\infty)=4$. Prop.~\ref{prp: ZZsum} implies that $ \sum_{i=1}^4 \frac{1}{2a_i +1} \geq 1$. Note that if $a_1=a_2=a_3=1$, then the corresponding sequence admits a packing total coloring of $\mathbb{Z}$ in which the first three colors have packing constraint equal to one. Hence, we may assume that $a_3 \ge 2$ (and hence $a_4 \ge 2$). 
    Considering the condition $ \sum_{i=1}^4 \frac{1}{2a_i +1} \geq 1$, we observe the following. 
    \begin{itemize}
        \item If $a_2 \geq 2$, then $\sum_{i=1}^4 \frac{1}{2a_i +1} \leq \frac{1}{3} + 3 \cdot \frac{1}{5} = \frac{14}{15} < 1$.
        \item If $a_3 \geq 3$, then $\sum_{i=1}^4 \frac{1}{2a_i +1} \leq 2 \cdot \frac{1}{3} + 2 \cdot \frac{1}{7} = \frac{20}{21} < 1$.
        \item If $a_4 \geq 4$, then $\sum_{i=1}^4 \frac{1}{2a_i +1} \leq 2 \cdot \frac{1}{3} + \frac{1}{5} + \frac{1}{9} = \frac{44}{45} < 1$.
    \end{itemize}

These observations imply that $a_2=1$, $a_3=2$ and $a_4 \in \{2,3\}$. 
Therefore, we only need to consider the sequences $S_1 = (1,1,2,2)$ and $S_2 = (1,1,2,3)$, which means that we have only two candidates for sequence $S$ such that $\chi_S^{''}(P_\infty)=4$. Clearly, by Theorem~\ref{infinite_path_known}, $\chi_{S_1}^{''}(P_\infty)=4$. Hence it remains to show that this is the only sequence that admits an $S$-packing total coloring of $P_\infty$ using $4$ colors.
   %
%
    Suppose to the contrary that there exists an $S_2$-packing total coloring $c$ of $P_\infty$ (using $4$ colors).
    Since by Prop.~\ref{prp: chi3}, $\chi_{S_2}^{''}(P_\infty) \neq 3$, there exists a vertex $x$ or an edge $e_x$ to which $c$ assigns a color $4$. Let $c(x)=4$ and distinguish three cases according to the colors assigned by $c$ to $e_x$ and $x+1$.
    
    \noindent \textbf{Case 1:} $c(e_x)=3$. \\
    In this case, $c(x+1), c(e_{x+1}), c(x+2) \in \{1,2\}$, a contradiction to $c$ being a proper $S_2$-packing total coloring of $\ZZ$ using $4$ colors. 
    
    \noindent \textbf{Case 2:} $c(e_x) \in \{1,2\}$ and $c(x+1)=3$. \\
    This implies $c(e_{x+1}), c(x+2), c(e_{x+2}) \in \{1,2\}$, a contradiction.
    
    \noindent \textbf{Case 3:} $c(e_x), c(x+1) \in \{1,2\}$. \\
    In this case, $c(e_{x+1})=3$, which together with the coloring of $x, x+1$ and $e_x$ further implies that $c(x+2), c(e_{x+2}), c(x+3) \in \{1,2\}$, again contradiction to $c$ being a proper $S_2$-packing total coloring of $P_\infty$.
    
  Note that if there does not exist a vertex in $V(P_\infty)$ to which $c$ assigns a color $4$, there is an edge $e_x$ with $c(e_x)=4$. In this case, analogously as above we derive a contradiction to the assumption that $c$ is a proper $S_2$-packing total coloring of $P_\infty$. Hence, we conclude that there is no $(1,1,2,3)$-packing total coloring of $P_\infty$ and consequently, $\chi_S^{''}(P_\infty)=4$ if and only if $S = S_1 = (1,1,2,2)$.
\qed    
\end{proof}

We continue by determining all sequences $S$ for which $\chi_S^{''}(P_\infty)=5$. Note that by Theorem~\ref{infinite_path_known}, $(1,2,2,2,2)$ and $(2,2,2,2,2)$ are among them.

\begin{theorem} \label{prp: ZZ5}
    Let $S=(a_1,a_2,a_3,a_4, a_5, \ldots)$. Then $\chi_S^{''}(P_\infty)=5$ if and only if $S$ is one of the following sequences: $(1,2,2,2,2)$, $(2,2,2,2,2)$, $(1,2,2,2,3)$, $(1,1,2,3,3)$, $(1,1,3,3,3)$, $(1,1,2,3,4)$, $(1,1,2,4,4)$, $(1,1,3,3,4)$, $(1,1,3,4,4)$, $(1,1,4,4,4)$. 
\end{theorem}

\begin{proof}
 Let $S = (a_1,a_2,a_3,a_4, a_5, \ldots)$ be an arbitrary non-decreasing sequence of positive integers such that $\chi_S^{''}(P_\infty)=5$. By Prop.~\ref{prp: ZZsum} we know that $ \sum_{i=1}^5 \frac{1}{2a_i +1} \geq 1$. We observe the following. If $a_1 \geq3$, then $\sum_{i=1}^5 \frac{1}{2a_i +1} \leq \frac{5}{7} <1$. This means that $a_1 \in \{1,2\}$. Next, if $a_2 \geq 3$, then $\sum_{i=1}^5 \frac{1}{2a_i +1} \leq \frac{1}{3} + \frac{4}{7} = \frac{19}{21}< 1.$ Hence, $a_2 \in \{1,2\}$. Now, we distinguish three cases regarding the values of $a_1$ and $a_2$. \\
\textbf{Case 1.} $a_1=a_2=2$ \\
If $a_3 \geq 3$, then $\sum_{i=1}^5 \frac{1}{2a_i +1} \leq \frac{2}{5} + \frac{3}{7} = \frac{29}{35}< 1$, which implies that $a_3=2$. Next, similarly we can conclude that $a_4=2$. Finally, if $a_5 \geq 3$, then $\sum_{i=1}^5 \frac{1}{2a_i +1} \leq \frac{4}{5} + \frac{1}{7}=\frac{33}{35} < 1$. Consequently, $a_5=2$ and we conclude that $S=(2,2,2,2,2)$. By Theorem~\ref{infinite_path_known}, in this case $\chi^{''}_S(P_\infty)=5$, hence $S=(2,2,2,2,2)$ is one of the appropriate sequences.
\\ 
\textbf{Case 2.} $a_1=1$, $a_2=2$ \\
If $a_3 \geq 3$, then $\sum_{i=1}^5 \frac{1}{2a_i +1} \leq \frac{1}{3}+ \frac{1}{5} + \frac{3}{7} = \frac{101}{105}< 1$, which implies that $a_3=2$. Next, $a_4 \geq 4$ implies $\sum_{i=1}^5 \frac{1}{2a_i +1} \leq \frac{1}{3}+ \frac{2}{5} + \frac{2}{9} = \frac{43}{45}< 1$, hence $a_4=2$ or $a_4=3$. We distinguish two sub-cases regarding the value of $a_4$.\\
\hspace*{1cm}\textbf{Case 2.1} $a_1=1$, $a_2=a_3=a_4=2$.  \\
If $a_5=2$, we have the sequence $S=(1,2,2,2,2)$ and Theorem~\ref{infinite_path_known} confirms that $\chi^{''}_S(P_\infty)=5$. Next, in the case when $a_5=3$, by coloring the consecutive elements of $P_\infty$ one after another using the pattern $[1, 2, 3, 1, 4, 5, 2, 1, 3, 4, 1, 2, 5, 3, 1, 4, 2, 1, 3, 5, 4]$, we form an $S$-packing total coloring of $P_\infty$ using $5$ colors. Hence, the sequence $S=(1,2,2,2,3)$ is one of the appropriate sequences.
Now, suppose that $a_5 \geq 4$ and that in this case $P_\infty$ admits an $S$-packing total coloring $c$ using $5$ colors. Since any two distinct elements of $P_\infty$, both colored by $5$ are at distance at least $5$, for some $i \in \ZZ$ we have the set $X$ containing the elements $i, e_i, i+1, e_{i+1}, i+2, e_{i+2}, i+3, e_{i+3}$ (note that the written sequence can also start with $e_i$ and end with $i+4$, but in this case the proof is analogous), to which $c$ assigns colors $1,2,3$ and $4$. We observe that at most three elements from $X$ can receive a color $1$ by $c$, while the colors $2$, $3$ and $4$ can be assigned each to at most $2$ elements from $X$. If $c(e_{i+1}), c(i+2) \in \{2,3,4\}$, then two of the colors from $\{2,3,4\}$ can be assigned each to only one element from $X$. Consequently, $c$ can assign colors to at most $7$ elements from $X$, but $|X|=8$, a contradiction. Next, if $c(e_{i+1})=1$, then $c$ must assign colors $2$,$3$ and $4$ to elements $e_i, i+1, i+2$ and $e_{i+2}$, which is not possible. Similarly we derive a contradiction also in the case when $c(i+2)=1$. Therefore, $c$ cannot be a proper $S$-packing total coloring of $P_\infty$ using $5$ colors.  
\\
\hspace*{1cm}\textbf{Case 2.2} $a_1=1$, $a_2=a_3=2$, $a_4=3$.  \\
If $a_5 \geq 4$, then $\sum_{i=1}^5 \frac{1}{2a_i +1} \leq \frac{1}{3}+ \frac{2}{5} + \frac{1}{7} + \frac{1}{9} = \frac{311}{315}< 1$. Hence $a_5=3$ and $S=(1,2,2,3,3)$. 
Suppose that in this case, $P_\infty$ admits an $S$-packing total coloring $c$ using $5$ colors. Note that any two distinct elements of $P_\infty$, both colored by $4$ are at distance at least $4$. This means that for some $i \in \ZZ$ we have the set $X$ containing the elements $i, e_i, i+1, e_{i+1}, i+2, e_{i+2}$ (note that the written sequence can also start with $e_i$ and end with $i+3$, but in this case the proof is analogous), to which $c$ assigns colors $1,2,3$ and $5$. We observe that only two elements from $X$ can receive a color $1$ by $c$, at most $2$ color $2$, at most $2$ color $3$ and at most one element color $5$. Additionally, $c$ assigns color $2$ or color $3$ to exactly $2$ elements from $X$. Without loss of generality suppose that two elements from $X$, namely $i$ and $e_{i+2}$ receive $2$ by $c$. Consequently, color $3$ can be assigned only to one element from $X$ and hence, color $1$ must be assigned to two elements of $X \setminus \{i, e_{i+2}\}$, namely $e_i$ and $i+2$. Finally, elements $i+1$ and $e_{i+1}$ receive colors $3$ and $5$ by $c$. Due to the symmetrical coloring we can say that $c(i+1)=3$ and $c(e_{i+1})=5$. Consequently, it is necessary that $c(e_{i-1})=c(i+3)=4$. If follows that $c(e_{i+3})=1$ or $c(e_{i+3})=3$. The first case yields that $c(i+4)=3$ and there is no color, which can be assigned by $c$ to $e_{i+4}$, a contradiction to $c$ being an $S$-packing coloring of $\ZZ$. In the second case it is necessary that $c(i+4)=1$, but also in this case there is no color, which can be assigned by $c$ to $e_{i+4}$. Hence, $c$ is not an $S$-packing coloring of $P_\infty$. \\
\textbf{Case 3.} $a_1=a_2=1$ \\
If $a_3 \geq 5$, then $\sum_{i=1}^5 \frac{1}{2a_i +1} \leq \frac{2}{3}+ \frac{3}{11} = \frac{31}{33}< 1$. Hence $a_3 \leq 4$. Similarly, we derive that $a_4 \leq 4$. Namely, if not, $\sum_{i=1}^5 \frac{1}{2a_i +1} \leq \frac{95}{99}<1$. Finally, if $a_5 \geq 5$, then $\sum_{i=1}^5 \frac{1}{2a_i +1} \leq \frac{97}{99}<1$. Therefore, $a_3, a_4, a_5 \leq 4$.
Note that for $S=(1,1,4,4,4)$ there exists an $S$-packing total coloring of $P_\infty$ using $5$ colors. Indeed, color the consecutive elements of $P_\infty$ one after another using the pattern $[1, 2, 3, 1, 2, 4, 1, 2, 5]$. This implies that in this case, $\chi^{''}_S(P_\infty)=5$. Moreover, using Observation~\ref{obs1} and Theorem~\ref{prp: ZZ4} we derive that $\chi^{''}_S(P_\infty)=5$ also in the cases when $S$ is one of the following sequence: $(1,1,3,4,4)$, $(1,1,3,3,4)$, $(1,1,2,3,4)$, $(1,1,2,4,4)$, $(1,1,3,3,3)$, or $(1,1,2,3,3)$. This concludes the proof. 
\qed    
\end{proof}

\section{$S$-packing total coloring of paths and cycles}

In this section, we determine the exact values of $S$-packing total chromatic numbers for all paths and cycles in the cases when the sequence $S$ contains only the integers $1$ and/or $2$.

Note that the case of $P_1$ is trivial, since $\chi_S^{\prime\prime}(P_1)=1$ for any sequence $S$. 
Next, from Prop. \ref{prp: chi3} we derive that $\chi_S^{\prime\prime}(P_2)=3$ for every sequence $S$.
Consequently, throughout this section we consider only paths $P_n$ and cycles $C_n$ with $n\geq 3$.

We denote the vertices of a path $P_n$ and a cycle $C_n$, where $n \geq 3$, by $v_1, v_2, v_3, \ldots, v_n$. Additionally, for any $i, j \in \{1,2,\ldots,n\}$, the edge between the vertices $v_i$ and $v_j$ is briefly denoted by $e_{i,j}$. In some proofs, we use the term \emph{consecutive elements} from $V(P_n) \cup E(P_n)$ or from $V(C_n) \cup E(C_n)$, which refers to consecutive elements in the sequence
$v_1$, $e_{1,2}$, $v_2$, $e_{2,3}$, $v_3$, $\ldots$, $v_n$ (and $e_{n,1}$ in the case of a cycle).


Based on Prop.~\ref{prp: chi3} we know that in the case when $S=(1^3,\dots)$, $\chi_S^{''}(P_n)=3$ for any $n \geq 2$. For this sequence $S=(1^3,\dots)$, with the following proposition, we consider the $S$-packing total chromatic number of cycles $C_n$.


\begin{proposition}
    Let $S=(1^3,2^{\infty})$ and $n \geq 3$. Then
    $$
    \chi_S^{''}(C_n)=\left\{
	\begin{array}{ll}
		3; \ n \equiv 0 \ (\mathrm{mod} \ 3),\\
		4; \ n \not \equiv 0 \ (\mathrm{mod} \ 3) \textit{ and } n \neq 4,\\
        5; \ n=4.
	\end{array}\right.
    $$
\end{proposition}

\begin{proof}
Let $S=(1^3,2^{\infty})$ and $n \geq 3$ an arbitrary integer. Prop.~\ref{prp: chi3} implies that $\chi_S^{\prime\prime}(C_n)=3$ if $n \equiv 0 \pmod{3}$, and $\chi_S^{\prime\prime}(C_n)>3$ otherwise. 

Now, let $n \not\equiv 0 \pmod{3}$ and $n \neq 4$. 
In order to prove that $\chi_S^{\prime\prime}(C_n) \leq 4$, we form an $S$-packing total coloring of $C_n$ as follows. 
In the case when $n \equiv 1 \pmod{3}$, color the elements from $V(C_n) \cup E(C_n)$ one after another by colors $1,2,3,1,2,3,4,1,2,3$, $1$, $2$, $3$, $4$ and, if necessary, the next remaining vertices color one after another by applying the pattern $[1,2,3]$.
If $n \equiv 2 \pmod{3}$, then color one vertex from the cycle $C_n$ by $4$ and the remaining elements from $V(C_n) \cup E(C_n)$ one after another using the following pattern of colors $[1,2,3]$. Note that the first presented coloring is a proper $S$-packing total coloring of a cycle $C_n$ for all $n \geq 7$, whereas the second one is valid for all $n \geq 5$. Hence, these colorings imply that $\chi_S^{\prime\prime}(C_n) =4$ if $n \not\equiv 0 \pmod{3}$ and $n \neq 4$.

It remains to consider the case when $n=4$. Since $V(C_4) \cup E(C_4)$ contains only elements which are pairwise at distance at most $2$, we derive that any $S$-packing total coloring of $C_4$ uses colors $1,2,3$ at most twice and any other color at most once. Then, the fact that $|V(C_4)|+|E(C_4)|=8$ implies that $\chi_S^{\prime\prime}(C_4) \geq 5$. By coloring the consecutive elements of $C_4$ by colors $1,2,3,1,2,3,4,5$, we form an $S$-packing total coloring of $C_4$ which proves that $\chi_S^{\prime\prime}(C_4) = 5$.
    \qed
\end{proof}

We continue with considering the sequence $S=(1,1,2^{\infty})$.

\begin{proposition}\label{prp:1^2,2^2 Pn}
    Let $S=(1^2,2^{\infty})$ and $n \geq 3$. Then
    $$
    \chi_S^{''}(P_n)=\left\{
	\begin{array}{ll}
		3; \ n=3,\\
		4; \ n \geq 4.
	\end{array}\right.
    $$
    \label{paths_1122}
\end{proposition}

\begin{proof}
    Let $S=(1^2,2^{\infty})$. From Prop.~\ref{prp: chi3} it directly follows that 
   $ \chi_S^{''}(P_3)=3.$
    Now, suppose that there exists an $S$-packing total coloring $c$ of $P_4$ using only $3$ colors. Since $P_4$ contains two disjoint subgraphs isomorphic to $P_2$ and $\chi_S^{''}(P_2) = 3$, $c$ assigns color $3$ to two distinct elements from $(V(P_4) \cup E(P_4)) \setminus \{v_2,v_3,e_{2,3}\}$. Then all elements from $\{v_2,v_3,e_{2,3}\}$ receive colors from $\{1,2\}$ by $c$, a contradiction to $c$ being an $S$-packing total coloring of $P_4$ using $3$ colors. 
    Hence, $\chi_S^{''}(P_4) \geq 4$, and moreover, by Observation~\ref{obs2}, $\chi_S^{''}(P_n) \geq 4$ for any $n \geq 4$. Finally, using Observation~\ref{obs2} and Theorem~\ref{infinite_path_known} we derive that $\chi_S^{''}(P_n)=4$ for any $n \geq 4$.


\qed
\end{proof}

\begin{proposition}
    Let $S=(1^2,2^{\infty})$ and $n \geq 3$. Then
    $$
    \chi_S^{''}(C_n)=\left\{
	\begin{array}{ll}
		4; \ n \notin \{4,7\},\\
        5; \ n=7,\\
        6; \ n=4.
	\end{array}\right.
    $$
\end{proposition}

\begin{proof}
    Let $S=(1^2,2^{\infty})$. First, consider the case when $n=4.$ The distance between any two elements from $V(C_4) \cup E(C_4)$ is at most $2$, hence any $S$-packing total coloring of $C_4$ uses each color $k\geq 3$ at most once, while colors $1$ and $2$ can be assigned each to at most two objects of $C_4$. Further, since $|V(C_4) \cup E(C_4)|=8$, $\chi_S^{''}(C_4) \geq 6$. By coloring the consecutive elements of $C_4$ by colors $1,2,3,1,2,4,5,6$, we form an $S$-packing total coloring of $C_4$ which proves that $\chi_S^{\prime\prime}(C_4) = 6$. 
    
    Now, consider a cycle $C_7$. Since $\diam(C_7)=4$, any $S$-packing total coloring of $C_7$ assigns colors $1$ and $2$ each to at most four elements of $C_7$, while any other color to at most two elements. Consequently, $\chi_S^{''}(C_7) \geq 5$. By applying the color pattern $[1,2,3,1,2,4,5]$ for coloring the consecutive elements of $C_7$, we obtain an $S$-packing total coloring of $C_7$ using $5$ colors. This means that $\chi_S^{''}(C_7) = 5$.
    
    Next, consider the case when $n \notin \{4,7\}$. It is clear from Prop.~\ref{prp: chi3} and Prop.~\ref{prp:1^2,2^2 Pn} that $\chi_S^{''}(C_n) \geq 4$. To show that $\chi_S^{''}(C_n) \leq 4$ for all $n\notin\{4,7\}$, we form an $S$-packing total coloring of each cycle $C_n$ using the following two color patterns: $c_1=[1,2,3,1,2,4]$ consisting of $6$ elements and $c_2=[1,2,3,1,4,2,1,3,2,4]$ consisting of $10$ elements. 
    We distinguish three cases with respect to the value of $n$.\\
    \indent \textbf{Case 1:} $n\equiv 0 \ (\mbox{mod } 3).$ \\
    In this case $|V(C_n)|+|E(C_n)|=6k$ and by applying the patten $c_1$ for coloring the consecutive elements of $C_n$, we form an $S$-packing total coloring of $C_n$ using $4$ colors.
    \indent \textbf{Case 2:} $n\equiv 1 \ (\mbox{mod } 3).$ \\
    Since $n \notin \{4,7\}$, $n=3k+1$ where $k \geq 3$. Therefore, $|V(C_n)|+|E(C_n)|=6k+2 =6(k-3)+20$. Hence by coloring the consecutive elements of $C_n$ by applying the pattern $c_2$ twice and then, if necessary, the pattern $c_1$, we obtain an $S$-packing total coloring of $C_n$.  \\
    \indent \textbf{Case 3:} $n\equiv 2 \ (\mbox{mod } 3).$ \\
    In this case, $n=3k+2$ and then, $|V(C_n)|+|E(C_n)|=6k+4 =6(k-1)+10$. This means that by applying the pattern $c_2$ once and $c_1$ more times if necessary, we obtain an $S$-packing total coloring of $C_n$ using $4$ colors, which concludes the proof.
\qed
\end{proof}

We now turn to the sequence $S = (1, 2^{\infty})$.

\begin{proposition}
    Let $S=(1,2^{\infty})$ and $n \geq 3$. Then
    $$
    \chi_S^{''}(P_n)=\left\{
	\begin{array}{ll}
        4; \ n=3,\\
		5; \ n \geq 4.
	\end{array}\right.
    $$
    \label{paths_1222}
\end{proposition}

\begin{proof}
    Let $S=(1,2^{\infty})$ Using Observation~\ref{obs2} and Prop.~\ref{prp: chi3}, we derive that 
    $\chi_S^{''}(P_3)>3$. In order to prove that $\chi_S^{''}(P_3) \leq 4$, color the consecutive elements of $P_3$ one after another using colors $1,2,3,1,4$. Indeed, this coloring is an $S$-packing total coloring of $P_3$ using $4$ colors, thus $\chi_S^{''}(P_3)=4$. 
    
    For $P_4$ suppose that there exists an $S$-packing total coloring $c$ using only $4$ colors. We observe that if $c$ assigns a color from $\{2,3,4\}$ to the element from $\{v_2, v_3, e_{2,3}\}$, then this color can be used only once by $c$. Consequently, one element from $\{v_2, v_3, e_{2,3}\}$ receives a color $1$ by $c$ (and the other two receive colors from $\{2,3,4\}$). If $c(v_2)$=1 or $c(v_3)=1$, then $P_4$ contains at most two elements colored by $1$. Additionally, there are at most two elements colored by one of the colors from $\{2,3,4\}$, while the other two colors are assigned each to at most one element of $P_4$. Therefore, $c$ assigns a color to at most $6$ elements of $P_4$, a contradiction since $|V(P_4)|+|E(P_4)|=7$. Consequently, $c(e_{2,3})=1$ and $c(v_1)=c(v_4)=1$. In this case, the remaining three colors are assigned by $c$ to the remaining four elements from $V(P_4) \cup E(P_4)$, a contradiction to $c$ being an $S$-packing total coloring of $P_4$. This means that $\chi_S^{''}(P_4)>4$ and consequently, $\chi_S^{''}(P_n)>4$ for any $n \geq 4$. Finally, using Observation~\ref{obs2} and Theorem~\ref{infinite_path_known} we derive that $\chi_S^{''}(P_n)=5$ for any $n \geq 4$.

\qed
\end{proof}

\begin{proposition}
    Let $S=(1,2^{\infty})$ and $n \geq 3$. Then
    $$
    \chi_S^{''}(C_n)=\left\{
	\begin{array}{ll}
		5; \ n \notin \{4,7\},\\
        6; \ n=7,\\
        7; \ n=4.
	\end{array}\right.
    $$
\end{proposition}

\begin{proof}
    Let $S=(1,2^{\infty})$. First, consider a cycle $C_3$. Since $|V(C_3)|+|E(C_3)|=6$ and any two distinct elements from $C_3$ are pairwise at distance at most $2$, any $S$-packing total coloring of $C_3$ assigns a color $1$ to at most two elements from $V(C_3) \cup E(C_3)$ and any other color to at most one element from the written set. Hence, $\chi_S^{''}(C_3) \geq 5$. On the other hand, by assigning a color $1$ to two elements from $V(C_3) \cup E(C_3)$ which are at distance $2$ (i.e., an edge and the opposite vertex) and four distinct colors to four remaining elements from $V(C_3) \cup E(C_3)$, we obtain an $S$-packing total coloring of $C_3$ using $5$ colors, which confirms that $\chi_S^{''}(C_3) = 5$.

    Next, let $n=4$. In this case $|V(C_4)|+|E(C_4)|=8$ and the distance between any two objects from $C_4$ is at most $2$. This means that any $S$-packing total coloring of $C_4$ assigns a color $1$ to at most two elements of $C_4$ and any other color to at most one element. Hence, $\chi_S^{''}(C_4) \geq 7$. By assigning a color $1$ to two non-adjacent vertices of $C_4$ and six distinct colors to the remaining elements, we obtain an $S$-packing total coloring of $C_4$ using $7$ colors, which confirms that $\chi_S^{''}(C_4) = 7$.

    Now, let $n \geq 5$. Since any $C_n$ contains a subgraph isomorphic to $P_n$, Observation~\ref{obs2} and Prop.~\ref{paths_1222} imply that $\chi_S^{''}(C_n) \geq 5$. To prove that $\chi_S^{''}(C_n) \leq 5$, we distinguish two cases with respect to the value of $|V(C_n)|+|E(C_n)|=2n$. \\
      \indent \textbf{Case 1.} $2n\equiv 0 \ (\mbox{mod } 5).$ \\
    By coloring the consecutive elements of $C_n$ one after another using the color pattern $[1,2,3,4,5]$ we obtain an $S$-packing total coloring of $C_n$ using $5$ colors, hence $\chi_S^{''}(C_n) = 5$. \\
      \indent \textbf{Case 2.} $2n\equiv p \ (\mbox{mod } 5)$, $p \in \{1,2,3,4\}$. \\
    In this case, color the consecutive elements of $C_n$ as follows: first, use the color pattern $[1,2,3,1,4,5]$ $p$-times and then (if necessary), continue with applying the pattern $[1,2,3,4,5]$. In this way, an $S$-packing total coloring of $C_n$ using $5$ colors is formed, hence $\chi_S^{''}(C_n) = 5$, except in the case when $n=7$. Indeed, if $n=7$, then $2n=14$ and $2n\equiv 4 \ (\mbox{mod } 5)$. Then, applying $[1,2,3,1,4,5]$ four times in not possible. Since for $C_7$, $2n=14$ and any two object are pairwise at distance at most $4$, any $S$-packing total coloring assigns a color $1$ to at most four elements and any other color to at most two elements. Hence, $\chi_S^{''}(C_7) \geq 6$. By coloring the consecutive elements of $C_7$ one after another using the pattern $[1,2,3,1,4,5,6]$, we obtain an $S$-packing total coloring of $C_7$ using $6$ colors, which confirms that $\chi_S^{''}(C_7) = 6$.
\qed
\end{proof}

Finally, the following two propositions address the case when $S=(2^{\infty})$.

\begin{proposition}
\label{prp_2222_paths}
    If $S=(2^{\infty})$, then
    $\chi_S^{''}(P_n)= 5$ for any $n \geq 3$.
\end{proposition}

\begin{proof}
    Let $S=(2^{\infty})$. First, consider a path $P_3$. Since any two elements from $V(P_3) \cup E(P_3)$ are pairwise at distance at most $2$ and $|V(P_3)| + |E(P_3)|=5$, we conclude that $\chi_S^{''}(P_3) \geq 5$. Consequently, $\chi_S^{''}(P_n) \geq 5$ for any $n \geq 3$. The upper bound follows directly from Theorem~\ref{infinite_path_known} and this concludes the proof. 
\qed
\end{proof}

\begin{proposition}
    Let $S=(2^{\infty})$ and $n \geq 3$. Then
    $$
    \chi_S^{''}(C_n)=\left\{
	\begin{array}{ll}
		5; \ n \equiv 0 \ (\mathrm{mod} \ 5),\\
		6; \ n \not \equiv 0 \ (\mathrm{mod} \ 5) \textit{ and } n \notin \{4,7\},\\
        7; \ n=7,\\
        8; \ n=4.
	\end{array}\right.
    $$
\end{proposition}

\begin{proof}
Let $S=(2^{\infty})$. In this proof we distinguish four cases with respect to the length $n$ of a cycle $C_n$.

\textbf{Case 1.}  $n=4$ \\
Since $|V(C_4)|+|E(C_4)|=8$ and any two object from $C_4$ are pairwise at distance at most $2$, any $S$-packing total coloring of $C_4$ assigns each color to at most one element from $V(C_4) \cup E(C_4)$. Therefore, any $S$-packing total coloring of $C_4$ uses $|V(C_4)|+|E(C_4)|=8$ colors, hence $\chi_S^{''}(C_4) = 8$.

\textbf{Case 2.}  $n=7$ \\
Note that any two objects from $C_7$ are pairwise at distance at most $4$, which implies that any $S$-packing total coloring of $C_7$ assigns each of the colors to at most two elements. Since $|V(C_7)|+|E(C_7)|=14$, $\chi_S^{''}(C_7) \geq 7$. By coloring the consecutive elements of $C_7$ one after another using the color pattern $[1,2,3,4,5,6,7]$ we form an $S$-packing total coloring of $C_7$ using $7$ colors, hence $\chi_S^{''}(C_7) = 7$.

\textbf{Case 3.}  $2n \equiv 0 \ (\mathrm{mod} \ 5)$ \\ 
Using Observation~\ref{obs2} and Prop.~\ref{prp_2222_paths} we derive that $\chi_S^{''}(C_n) \geq 5$ for any $n \geq 3$. On the other hand, by coloring the consecutive elements of $C_n$ one after another using the color pattern $[1,2,3,4,5]$ we obtain an $S$-packing total coloring of $C_n$ using $5$ colors, hence $\chi_S^{''}(C_n) = 5$. 

\textbf{Case 4.}  $2n\equiv p \ (\mbox{mod } 5)$, $p \in \{1,2,3,4\}$, and $n \notin \{4,7\}$ \\ 
Let $n$ satisfy the conditions $2n \equiv p \pmod 5$, where $p \in \{1,2,3,4\}$, and $n \notin {4,7}$. 

Color the consecutive elements of $C_n$ successively as follows. First, apply the color pattern $[1,2,3,4,5,6]$ exactly $p$ times. Then, if necessary, continue by repeating the pattern $[1,2,3,4,5]$ until all remaining elements have been colored. One can verify that this construction yields an $S$-packing total coloring of $C_n$, except when $2n \in \{7,8,9,13,14,19\}$. Since $2n$ is even and, by assumption, $2n \notin \{8,14\}$, none of these exceptional cases can occur. Therefore, the proposed coloring is a proper $S$-packing total coloring of $C_n$ using six colors. Consequently, $\chi_S^{''}(C_n)\leq 6$.

Now, suppose that $\chi_S^{''}(C_n) \leq 5$, and let $c$ be an $S$-packing total coloring of $C_n$ using five colors. Since any five consecutive elements of $C_n$ (either three vertices and two edges, or two vertices and three edges) are pairwise at distance at most $2$, they must receive five distinct colors by $c$. Without loss of generality, assume that
$c(v_1)=1$, $c(e_{1,2})=2$, $c(v_2)=3$, $c(e_{2,3})=4$, and $c(v_3)=5$.
Then the colors of all subsequent elements are uniquely determined, forcing the repeating pattern $[1,2,3,4,5]$ on $e_{3,4}, v_4, e_{4,5}, v_5, \ldots$ However, since $2n \not\equiv 0 \pmod 5$, this pattern cannot be completed consistently around the cycle. Consequently, $c$ cannot be a proper $S$-packing total coloring using only five colors, which contradicts to our assumption. Therefore, $\chi_S^{''}(C_n)=6$.

{}
    \qed
\end{proof}

\section{Open problems}
The introduction of $S$-packing total coloring and the initial results presented in this paper naturally lead to several open questions and avenues for further research. 

In the first part of the paper, we observed that $max\{\chi_S(G), \chi_S'(G)\} \leq \chi_S^{''}(G)$ for every graph $G$ and every sequence $S$. While the gap between $\chi_S(G)$ and $\chi_S''^{}(G)$ can be arbitrarily large, we also provided an example showing that equality may occur. Since this has been established for only one particular sequence $S$, namely $S=(1,1,1,\ldots)$, a natural question is to characterize the sequences $S$ and graphs $G$ for which $\chi_S(G) = \chi_S^{''}(G)$. Additionally, a difference $\chi_S^{''}(G)-\chi_S'(G)$ has not yet been considered.

Furthermore, for an arbitrary sequence $S$, we have characterized all graphs $G$ for which $\chi_S^{''}(G) \in \{1,2,3\}$. It would be natural to extend these results by providing a characterization of graphs with $S$-packing total chromatic number equal to
$4$ or, more generally, 
larger values. We note that Ferme and Mesari\v c \v Stesl have already provided a characterization of graphs with $S$-packing total chromatic number $4$ or $5$ in the special case when $S=(1,2,3,\ldots)$\cite{ferme-stesl}. 

Moreover, for any $S$, it would be natural to determine in which families of graphs the $S$-packing total chromatic number is unbounded. For instance, since the packing chromatic number in the family of (sub)cubic graphs is unbounded~\cite{balogh-2018}, the same holds also for $S$-packing total chromatic number of graphs when $S=(1,2,3, \ldots)$. 

Next, for an arbitrary sequence $S$ of positive integers, we determined the $S$-packing total chromatic numbers for all complete bipartite graphs. Moreover, for sequences $S$ containing only the integers $1$ and/or $2$, we established the values of $\chi_S^{''}(G)$ when $G$ is a finite or infinite path, or a cycle. The study of $S$-packing total colorings for other classical families of graphs remains open.

Furthermore, it is not known how the $S$-packing total chromatic number behaves under standard graph operations such as the Cartesian product, the corona product, or the Mycielskian construction, in terms of the $S$-packing total chromatic numbers of input graphs. It would also be interesting to study $\chi_S^{''}(G)$ with respect to various local operations on a graph $G$, such as edge subdivision, edge contraction, or edge and vertex deletion. Note that the last two types of local modification are closely related to the concept of critical graphs.
For any sequence $S$, a graph $G$ is called \textit{$S$-packing critical} if $\chi_S(H)<\chi_S(G)$ for every proper subgraph $H$ of $G$~\cite{bujtas}. Similarly, we can say that a graph is \textit{$S$-packing total critical} if $\chi_S^{''}(H)<\chi_S^{''}(G)$ for any proper subgraph $H$ of $G$. It is natural to consider such graphs, for example in order to characterize $S$-packing total critical graphs satisfying $\chi_S^{''}(G)=k$ for a given choice of sequence $S$.

\section{Acknowledgements}
J.F. acknowledges the financial support from the Slovenian Research Agency (N1-0431).


\end{document}